\documentclass{amsart}
\usepackage{amsfonts,latexsym,graphicx,amssymb,amsmath,url,amsthm}
\newtheorem{theorem}{Theorem}
\newtheorem{lemma}{Lemma}
\newtheorem{corollary}{Corollary}
\newtheorem{proposition}{Proposition}
\newtheorem{conjecture}{Conjecture}

\theoremstyle{definition}
\theoremstyle{remark}
\newtheorem*{remark}{{\bf{Remark}}}
\newcommand{\N}{\mathbb N}
\newcommand{\Z}{\mathbb Z}

\begin{document}

\title{
Hilbert cubes in arithmetic sets}

\author{Rainer Dietmann}
\address{Department of Mathematics,
Royal Holloway, University of London, Egham, TW20 0EX Surrey, UK}
\email{Rainer.Dietmann@rhul.ac.uk}
\author{Christian Elsholtz}
\address{Institut f\"ur Analysis und Computational Number Theory,
Technische Universit\"at Graz,
Steyrergasse 30,
A-8010 Graz, Austria}
\email{elsholtz@math.tugraz.at}

\date{\today}
\begin{abstract}
We show upper bounds on the maximal dimension $d$
of Hilbert cubes 
$H=a_0+\{0,a_1\}+\cdots + \{0, a_d\}\subset S \cap [1, N]$ 
in several sets $S$ of arithmetic interest.

\begin{itemize}
\item[a)]
For the set of squares
we obtain  $d=O(\log \log N)$. Using previously known methods this bound
could have been achieved only conditionally subject to an unsolved 
problem of Erd\H{o}s and Rado.
\item[b)]
For the set $W$ of powerful numbers we show $d=O((\log N)^2)$.
\item[c)]
For the set $V$ of pure powers we also show $d=O((\log N)^2)$,
but for a homogeneous Hilbert cube, with $a_0=0$,
this can be improved to 
$d=O((\log \log N)^3/\log \log \log N)$,
when the $a_i$ are distinct, and 
$d=O((\log \log N)^4/(\log \log \log N)^2)$,
generally. 
This compares with a result of
$d = O((\log N)^3/(\log \log N)^{1/2})$ in the literature.
\item[d)] For the set $V$ we also solve an open problem of
Hegyv\'ari and S\'ark\"ozy, 
namely we show that $V$ does not contain an infinite Hilbert cube.
\item[e)]
For a set without arithmetic progressions of length $k$ we prove
$d=O_k(\log N)$, which is close to the true order of magnitude.
\end{itemize}
\end{abstract}
\thanks{Acknowledgement: 
C. Elsholtz was partially supported by FWF-DK Discrete Mathematics Project W1230.}
\maketitle
\section{Introduction}
\subsection{Results on progression-free sets and squares,
discussion of methods}
If $a_0 \ne 0, a_1, \ldots, a_d$ are elements of an additive group, we define
\begin{align*}
   H(a_0;a_1, \ldots , a_d) & :=a_0+\{0,a_1\}+\cdots + \{0, a_d\} \\
   & = \left\{ a_0 + \sum_{i=1}^d \varepsilon_i a_i : \varepsilon_i \in \{0,1\}
   \right\}
\end{align*}
to be
a Hilbert cube of dimension $d$. Note that here and in the following
we do not require the $a_i$ to be distinct, unless noted explicitly.
For $a_0=0$ it is convenient to slightly amend the definition to
\[
   H(0;a_1, \ldots , a_d)  :=
   \left\{ \sum_{i=1}^d \varepsilon_i a_i : \varepsilon_i \in \{0,1\},
   \quad \sum_{i=1}^d \varepsilon_i>0
   \right\},
\]
excluding the empty sum. This special case of a homogeneous Hilbert cube
is also often called the set of `subset sums'. Moreover, one can in an
obvious way extend these definitions to infinite Hilbert cubes, imposing
the additional condition that only finitely many $\varepsilon_i$ are different
from zero. Homogeneous infinite Hilbert cubes are also known in the literature 
as `finite sums', or as `IP sets' (see for example \cite{Bergelson:2010}). 
For an infinite set $A$ one writes this as
\[FS(A)=  \left\{ \sum_{i} \varepsilon_i a_i : \varepsilon_i \in \{0,1\},
   \quad 0< \sum_{i} \varepsilon_i < \infty \right\}.\]
In this paper, though, our focus will mainly, but not
exclusively, be on finite Hilbert cubes.

There are many questions connecting the size $d$ with the number of 
distinct elements in a $d$-dimensional cube. This topic 
is closely connected to questions on `sumset growth'
studied in additive combinatorics.
The question of the maximal size $d$ such that there is some Hilbert cube 
$H(a_0;a_1, \ldots , a_d)$ 
in a given set $S$ of integers has been 
frequently studied. Also, a number of different methods have been used
in these investigations.

Given a set $S \subset \{1, \ldots, N\}$
and
\[
  H(0; a_1, \ldots, a_{d}) \subset S
\]
for positive $a_i$, then also
\[
  H(a_1; a_2, \ldots, a_d) \subset S,
\]
so the maximal dimension of a subset sum
in $S$ clearly is less than or equal to one plus the maximal dimension of a
Hilbert cube in $S$.
In general, it seems, the case of subset sums might be easier.
At least, certain methods developed for subset sums, 
see e.g. \cite{Csikvari:2008}, cannot be extended to the general case.

In this paper we focus on questions, where the set $S$ has some interesting
arithmetic meaning,  
such as the set of squares, or the set of powerful numbers.
Here the results and methods by 
Hegyvari and S\'{a}rk\"{o}zy \cite{HegyvariandSarkozy:1999}
and Gyarmati, S\'{a}rk\"{o}zy and Stewart \cite{GSS:2003} have been influential.
As we will observe,
the study of these arithmetic sets
 is closely related to questions on sets $S$ which do
not contain long arithmetic progressions.
There are a number of investigations, connecting sumsets and subset sums
with the length of arithmetic progressions or studying sumsets in the set of
squares. We will mention some of these results below,
but there is further work by 
Lagarias, Odlyzko and Shearer \cite{LagariasandOdlyzkoandShearer:1983},
S\'{a}rk\"{o}zy \cite{Sarkozy:1994}, Erd\H{o}s and S\'{a}rk\"ozy
\cite{ErdosandSarkozy:1992}, Ruzsa \cite{Ruzsa:1991},
 Szemer\'{e}di and Vu \cite{SzemerediandVu:2006, SzemerediandVu:2006JAMS},
Khalfalah and Szemer\'{e}di \cite{KhalfalahandSzemeredi:2006},
Nguyen and Vu \cite{NguyenandVu:2010}.

Further, the study of combinatorial aspects of Hilbert cubes 
has a long tradition, see for example Hilbert \cite{Hilbert:1892},
Hegyv\'{a}ri, \cite{Hegyvari:1996, Hegyvari:1999},   
Gunderson, R\"{o}dl and Sidorenko \cite {GundersonandRodlandSidorenko:1999},
and the very recent result by Conlon, Fox and Sudakov \cite{ConlonFoxSudakov}.

A connection of the maximal dimension of a Hilbert cube with the
complexity of sequences has been investigated by Woods \cite{Woods:2004}.
For example he showed that an upper bound on the dimension of a
Hilbert cube in the set of primes provides a lower bound on the complexity
of prime testing. While we do not follow this path here,
it can be observed that the existence of a high dimensional
Hilbert cube, say of dimension $d$, and of size $|H|> c^d$, with $c>1$,
(compare Lemma \ref{growth}) 
would allow to store these more than $c^d$ elements very efficiently, 
just by storing
the $d+1$ base elements. From this perspective, and in view of the philosophy
behind the famous conjecture of Erd\H{o}s and Szemer\'{e}di 
\cite{Erdos-Szemeredi:1983} that sets of
integers do not simultaneously contain a large additive and a large 
multiplicative structure, it seems reasonable to expect that for 
multiplicatively defined sets, such as those studied in this paper,
 such an efficient (additively defined) saving is
only possible for quite small subsets.

The connection of Hilbert cubes and arithmetic progressions
has been famously studied 
by Szemer\'{e}di \cite{Szemeredi:1969} and more recently by
Gowers \cite{Gowers:2011}, who introduced in this
connection what is now known as the Gowers uniformity norm,
and in the paper by Green and Tao \cite{GreenandTao:2008}
on long progressions in the primes.

In this paper we make use of a variety of methods, some of which have been
used before in connection with subset sums, but had to be adapted in order to
make use of them for the more general Hilbert cubes.
In our earlier paper \cite{DietmannandElsholtz:2012}
we had introduced a new method, based on the study of sumset growth for
progression-free sets. This led to an almost exponential sumset growth.
In this paper we introduce a tool, Lemma \ref{growth} below,
to the study of Hilbert cubes,  
which extends a recent method of Schoen \cite{Schoen:2011}, and which
eventually achieves exponential sumset growth 
for the sets under consideration.
For the set of squares and progression-free sets
(Theorems \ref{thm:squares} and \ref{thm:kprogression})
this improves, and simplifies, our work in \cite{DietmannandElsholtz:2012}.
The two methods above make use of global properties of the set
$S$, such as avoiding arithmetic progressions.

To handle the more delicate and less structured set of pure powers
it seems best to make in addition use of local obstructions modulo primes.
Here a number of estimates on primes, 
including Linnik's theorem and sieve methods, have to be used.
This route had been developed by Gyarmati, S\'{a}rk\"{o}zy and Stewart 
\cite{GSS:2003} for the study of subset sums.

The combination of these methods seems stronger than any of these methods alone.
After these prolegomena on methods let us turn our attention to the arithmetic
sets under consideration.

Brown, Erd\H{o}s and Freedman \cite{BrownandErdosandFreedman:1990}
asked whether the maximal
dimension of a Hilbert cube in the set
\[
  S_2 = \{n^2 : n \in \mathbb{N}\}
\]
of integer squares
is absolutely bounded or not. Experimentally, one finds only very small cubes
such as 
\[
   1+\{0,528\}+\{0,840\}+\{0,840\}
   = \{1^2, 23^2, 29^2, 29^2, 37^2, 37^2, 41^2, 47^2\}.
\]
Observe that in this example
$29^2$ and $37^2$ occur as sums in two different ways.
Cilleruelo and Granville \cite{CillerueloandGranville:2007},
Solymosi \cite{Solymosi:2007} and Alon, Angel, Benjamini and Lubetzky
\cite{AlonandAngelandBenjaminiandLubetzky:2012}
explain that the Bombieri-Lang conjecture 
implies that $d$ is absolutely bounded.
The study of sumsets $A+B$ being a subset of the set of integer squares 
$S_2\cap [1, N]$ is well understood when $|A|=2$. Here $a_2+b_i=x^2$ and
$a_1+b_i=y^2$ gives $a_2-a_1=x^2-y^2=(x+y)(x-y)$.
On the other hand, for any factorization $a_2 - a_1=d_1 d_2$ with $d_1>d_2>0$
of the same parity, one finds
$x=\frac{d_1+d_2}{2}, y=\frac{d_1-d_2}{2}$, and hence
$b_i=x^2-a_2=y^2-a_1$.
Given $a_1, a_2$ and the factorizations of $a_2-a_1$ one finds the 
corresponding set $B$.
For a fixed set $A=\{a_1, a_2\}$ with $a_1 \ne a_2$
there exist only finitely many 
possible values of corresponding $x,y$ and therefore only a finite set $B$.
On the other hand, when the elements $a_1,a_2$ can vary in $[1,N]$,
the size of $|B|$ can be as large as the divisor function 
$\frac{\tau(a_2-a_1)}{2}$, i.e.
\[
  |B| \leq \exp \left((\log 2 +o(1))\frac{\log N}{\log \log N}\right)
\]
is essentially the best possible general bound
(see \cite{RSS}, Theorem 4 and its proof).
When $|A|=3$ the problem of studying $A+B\subset S_2$ goes back to Euler.
For some historical explanation, and for relating the problem to elliptic
curves see Alon, Angel, Benjamini and Lubetzky
\cite{AlonandAngelandBenjaminiandLubetzky:2012}, for some extension see 
Dujella and Elsholtz \cite{DujellaandElsholtz:2013}.
The case of $|A| \geq 4$ leads to curves of genus $g\geq 2$,
which is the reason why the Bombieri-Lang conjecture is related to it.
We do not follow this path here.
Hegyv\'ari and S\'ark\"ozy
(\cite{HegyvariandSarkozy:1999}, Theorem 1) proved that for the set of 
integer squares
$S_2\cap [1, N]$ the maximal dimension of a Hilbert cube
is bounded by $d=O((\log N)^{1/3})$.
In a previous paper (\cite{DietmannandElsholtz:2012}, Theorem 3) we
improved this to $d=O((\log \log N)^2)$. Here we further reduce that bound.

\begin{theorem}[Hilbert cubes in the squares]{\label{thm:squares}}
Let $S_2$ denote the set of integer squares.
Let $N$ be sufficiently large, let 
$a_0$ be a non-negative integer and
let $A=\{a_1, \ldots, a_d\}$ be a set of positive integers
such that 
$H(a_0; a_1, \ldots , a_d)\subset S_2\cap [1,N]$. Then
\[
  d \leq 7 \log \log N.
\]
\end{theorem}
A comparable bound was proved in the authors' earlier paper
(\cite{DietmannandElsholtz:2012}, Theorem 1)
for $k$-th powers ($k\geq 3$) instead of squares. The treatment of higher powers
was easier for the following reason:
By a deep theorem of Darmon and Merel \cite{DarmonandMerel:1997},
following the proof of Fermat's last
theorem, for $k \ge 3$ there are no 3-progressions in the set of $k$-th powers.
On the contrary, the set of squares does contain arithmetic 3-progressions.
We should remark that
the special case $a_0=0$ of subset sums was previously  
studied by Csikv\'{a}ri (\cite{Csikvari:2008}, Corollary 2.5), who proved
in this case the same bound $d=O( \log \log N)$. His method of proof 
would not extend to the general case of $a_0 \neq 0$. 

\begin{remark}

With regard to our earlier result,
Noga Alon kindly pointed out to us that
Lemma 5 of \cite{DietmannandElsholtz:2012},
which corresponds to Lemma \ref{progressionfreelemma} in this paper, 
is actually a version of a result of Erd\H{o}s and Rado 
\cite{ErdosandRado:1960} on $\Delta$-systems (or sunflowers).
On this subject small quantitative
improvements are due to Kostochka \cite{Kostochka:1996}, 
even though the explicit dependence on the parameters
$h$ and $v$ is not well understood, and apparently at least one of the
parameters $h$ and $v$ is fixed.
For the set $S_2$ of squares, this would, assuming uniformity,
possibly lead to the tiny improvement
$d=O((\log_2 N)^2\frac{\log_5 N}{\log_4 N})$, 
the $\log_i N$ denoting the $i$-fold iterated logarithm.
Moreover, the Erd\H{o}s-Rado
conjecture on these $\Delta$-systems,
for which Erd\H{o}s \cite{Erdos:1981} offered a prize of $\$$1000,
would have implied $d=O(\log \log N)$.
Fortunately, it was possible to bypass the realm
of $\Delta$-systems and to prove for Hilbert cubes in the set of squares
the upper bound $d=O(\log \log N)$ unconditionally.
\end{remark}
\begin{corollary}[Hilbert cubes in quadratic polynomials]{
\label{cor:quadraticpoly}}
Let $f(x)=ax^2+bx+c$ be a quadratic polynomial where $a,b,c \in \Z$
such that $a>0$ and $4a+b^2-4ac \ge 1$,
and let $S=\{f(x) : x \in \Z\}$.
Let $a_0$ be a non-negative integer and $A=\{a_1, \ldots, a_d\}$ be a
set of positive integers such that $H(a_0; a_1, \ldots, a_d) \subset
S \cap [1,N]$. Then for sufficiently large $N$, we have
\[
  d \le 7 \log \log N.
\]
\end{corollary}

Note that the assumption $4a+b^2-4ac \ge 1$ is not really necessary,
but simplifies the proof as it avoids possibly hitting the number zero
which is excluded in our definition of the set $S_2$ of squares.\\
The method of proof allows us also to prove the following theorem, which
improves \cite[Theorem 2]{DietmannandElsholtz:2012}. The
special case of subset sums has
been studied by Schoen \cite[Corollary 2.2]{Schoen:2011}.

\begin{theorem}[Hilbert cubes in progression-free sets]
{\label{thm:kprogression}}
Let $k\geq 3$ be a positive integer, and
let $S$ denote a set of integers without an arithmetic progression of 
length $k$. Moreover, let
\[
  c = \frac{k}{k-1}.
\]
Then, for sufficiently large $N$, the following holds true: If 
$a_0$ is a non-negative integer and
$A=\{a_1, \ldots, a_d\}$ is a set of positive integers
such that 
$H(a_0; a_1, \ldots , a_d)$ $\subset S\cap [1,N]$, then
\begin{align*}
   d & \leq \frac{2(k-2)}{(k-1)\log c}\log N+\frac{2}{\log c}+1.
\end{align*}
\end{theorem}
\begin{remark}
A series expansion of $\log(1+\frac{1}{k-1})$ shows that
\[\frac{2(k-2)}{(k-1)\log c} 
=(2k-3) -\frac{7}{6k}-\frac{11}{12k^2} +O(\frac{1}{k^3}) .\]
\end{remark}

\begin{remark}
The upper bound in terms of $N$ is close to the correct order of magnitude, 
as can be seen from the following example, 
which is a variant of an example by Szekeres 
mentioned by Erd\H{o}s and Tur\'{a}n \cite{ErdosandTuran:1936}.

Let $k$ be a fixed prime.
In base $k$, we study the set of integers avoiding the digit $k-1$.
Let
\[
  H=H(0; a_1, \ldots , a_1, \ldots, a_s, \ldots, a_s),
\]
with
\[
  a_{i} =k^{i-1}: i=1, \ldots, s= \left \lfloor \frac{\log N}{\log k}
  \right \rfloor,
\]
where all elements occur $(k-2)$-fold.
Then all $n\in H$ can be written 
as $n= \sum_{i=0}^{s-1} b_i k^i$, with $b_i \in [0,k-2]$.
This set does not contain any arithmetic progression of length $k$ and
gives $d\sim \frac{k-2}{\log k} \log N$.

\end{remark}

The following lemma
proved to be the key idea and may be of independent interest;
some variant of it is 
implicitly contained in Schoen (\cite{Schoen:2011}, Lemma 2.1).

\begin{lemma}[A cube lemma for progression-free sets,
exponential growth]\label{growth}
Let $k\geq 3$ be a positive integer, and
let $S$ denote a set of integers without an arithmetic progression of 
length $k$. Moreover, let $a_0$ be an integer, let
$A=\{a_1, \ldots, a_d\}$ be a set of non-zero integers,
and let $H(a_0; a_1, \ldots, a_d) \subset S$. 
Then
\[
  |H|\geq 2 \left( \frac{k}{k-1} \right)^{d-1}-1.
\]
\end{lemma}
It is well known that for sets
without progressions of length $k=3$ one even has that $|H|=2^d$ for
$a_0 \ne 0$.
This is an exercise in Solymosi \cite{Solymosi:2007}, see also Lemma 3
in \cite{DietmannandElsholtz:2012}.

\subsection{Applications to pure powers and powerful numbers}
As an immediate application one can also improve upon a
result of
Gyarmati, S\'{a}rk\"{o}zy, and Stewart (\cite[Theorem 6]{GSS:2003}), where they
studied the maximal dimension $d$ of subset sums in the set of powerful numbers
\[W=\{n \in \N: p\mid n \Rightarrow p^2\mid n\}.\]
They proved an upper bound of $d=O((\log N)^3 (\log \log N)^{-1/2})$.
Note that they study subset sums only,
but their method of proof also seems
to give the same result for Hilbert cubes.
For this problem they also prove a lower bound of
$d \gg (\log N)^{1/2}$ for the maximal dimension
which shows that this problem is indeed quite different from
the case of Hilbert cubes in the set of squares studied above.

Observing that the powerful numbers cannot contain very long arithmetic
progressions, and making use of the exponential growth
guaranteed by Lemma \ref{growth}, we can establish the following result.

\begin{theorem}[Hilbert cubes in powerful numbers]{\label{thm:Hilbert-powerful}}
Let $W$ denote the set of powerful numbers.
Further, let $a_0$ be a non-negative integer, and let
$A=\{a_1, \ldots, a_d\}$
be a set of positive integers such that $H(a_0; a_1, \ldots, a_d)
\subset W \cap [1,N]$. Then
\[
  d \leq 5(\log N)^2
\]
for sufficiently large $N$.
\end{theorem}

Finally, we study Hilbert cubes and subset sums in the set 
\[V=\{a^n: a,n \in \N, n \geq 2\}\] 
of pure powers.
The set of pure powers behaves
much more irregularly than the set of powers of a fixed base $a$,
or the set of $n$-th powers for fixed exponent
$n$. There are some interesting 
recent results on the set of pure powers,
see for example \cite{GyoriHajduPinter:2009}.

As $V\subset W$, we obtain as a corollary to Theorem 
\ref{thm:Hilbert-powerful} the following result.
\begin{corollary}[Hilbert cubes in pure powers]{\label{cor:Hilbert-pure}}
Let $V$ denote the set of pure powers. Further,
let $a_0$ be a non-negative integer, and let $A=\{a_1, \ldots, a_d\}$
be a set of positive integers such that $H(a_0; a_1, \ldots, a_d)
\subset V \cap [1,N]$. Then
\[
  d \leq 5(\log N)^2
\]
for sufficiently large $N$.
\end{corollary}
The last two results are an application of Theorem
\ref{thm:kprogression} with $k=O(\log N)$. Restricting to subset sums,
we achieve a considerable improvement on this bound:
indeed we show that the relevant homogeneous progressions
are of the much shorter size $O(\frac{\log \log N}{\log \log \log N})$.
Using this, we are able to adapt the
method that we presented in \cite{DietmannandElsholtz:2012} to this situation.
It might seem that the method based on Lemma \ref{growth} is more
powerful, as the exponential growth is better than the growth 
established in step 3 of the proof of Theorem 4
in \cite{DietmannandElsholtz:2012}, see
formula \eqref{ineq1} below.
However, the previous method seems to be more flexible
in situations where only strong
bounds on the length of homogeneous
progressions rather than all progressions are available.
This way we obtain the following result.

\begin{theorem}[Subset sums in pure powers]
\label{brunnen}
Let $V$ denote the set of pure powers.
Further, let $A=\{a_1, \ldots, a_d\}$
be a set of positive integers such that $H(0; a_1, \ldots, a_d)
\subset V \cap [1,N]$. If the $a_i$ are pairwise distinct, then
\[
  d \ll \frac{(\log \log N)^3}{\log \log \log N}
\]
for sufficiently large $N$; in the general case, i.e. where the $a_i$ are not
necessarily distinct,
\[
  d \ll \frac{(\log \log N)^4}{(\log \log \log N)^2}
\]
holds true.
\end{theorem}
Note that we first assume here that the $a_i$ are distinct, as in the
course of proof we will invoke Lemma \ref{progressionfreelemma}
and Lemma \ref{vancouver}; otherwise, we
have to allow for another factor $O(\log \log N/\log \log \log N)$
in the upper bound for $d$, as this is the maximal multiplicity of
the elements $a_i$.\\
We are not aware of any previous result in the literature on this problem,
other than the above mentioned bound by
Gyarmati, S\'{a}rk\"{o}zy, and Stewart (\cite[Theorem 6]{GSS:2003}):
their upper bound
$d=O((\log N)^3 (\log \log N)^{-1/2})$ for the dimension of subset sums in $W$
 remains valid for subset sums in $V$. 

While it is easy to see that for any fixed exponent $k$ the set of $k$-th
powers cannot contain a Hilbert cube of infinite dimension,
Hegyv\'{a}ri and S\'{a}rk\"{o}zy 
\cite[page 314]{HegyvariandSarkozy:1999}
mention that it is
an open question whether or not 
there exists an infinite Hilbert cube in the set of pure powers.
They relate this question to the widely open Pillai conjecture
stating that a fixed integer
can be the gap between pure powers only finitely often.
Here we answer the question by
Hegyv\'{a}ri and S\'{a}rk\"{o}zy on Hilbert cubes in the set of pure powers, 
and with hindsight the problem was not as
difficult as anticipated.
\begin{theorem}{\label{infinitecubeV}}
There is no infinite Hilbert cube in the set $V$ of pure powers.
\end{theorem}
For the set of powerful numbers this question seems more 
difficult. There are, for example, infinitely many
pairs of consecutive powerful numbers.
In fact solutions of the Pell equation $x^2-8y^2=1$ provide examples.
However, one can give conditional results using either of the following
two well known conjectures.
\begin{conjecture}[$abc$ conjecture \cite{Masser}]
Let $\varepsilon>0$.
Suppose that $a+b=c$ for non-zero integers
$a,b,c$ with $(a,b,c)=1$, and let
\[
  P = \prod_{p \, | \, abc} p,
\]
the product taken over all primes $p$ dividing $abc$. Then
\[
  \max\{|a|, |b|, |c|\} \ll_\varepsilon P^{1+\varepsilon}.
\]
\end{conjecture}
\begin{conjecture}[Schinzel-Tijdeman conjecture \cite{ST}]
\label{red2}
If a polynomial $P$ with rational coefficients has 
at least three simple zeros, then the equation $y^2z^3=P(x)$ has only
finitely many solutions in integers $x,y,z$ with $yz\neq 0$.
\end{conjecture}
\begin{theorem}\label{infinitecubeW}
Assuming the $abc$ conjecture, or assuming the Schinzel-Tijdeman
conjecture, 
there is no infinite Hilbert cube in the set $W$ of powerful numbers.
\end{theorem}
Let us remark that Walsh \cite{Walsh:1999} proved that the $abc$ conjecture
implies the Schinzel-Tijdeman conjecture, but below we give two independent
and quite short proofs for Theorem \ref{infinitecubeV}, using either
the $abc$ conjecture or the Schinzel-Tijdeman conjecture.

\textbf{Acknowledgement}: 
We would like to thank Noga Alon for
 pointing out to
us the connection to the work of Erd\H{o}s and Rado
\cite{ErdosandRado:1960}.
In an attempt to avoid their
conjecture we got the idea of the present approach.
We would also like to thank J\"org Br\"udern for a comment
regarding  Lemma \ref{growth}, and the referee.

\section{Proofs of Theorems \ref{thm:squares} and 
\ref{thm:kprogression}}
The following observation is key for proving Lemma \ref{growth}.
\begin{lemma}\label{dense-progression}
Let $0<\alpha<1$, let $h$ be a non-zero integer and let $B$ be a non-empty
set of distinct
integers. If $|B \cap (B+h)|> (1-\alpha) |B|$, 
then $B$ contains an arithmetic progression
of length $\lfloor \frac{1}{\alpha} \rfloor +1$ and difference $h$.
\end{lemma}
\begin{proof}[Proof of Lemma \ref{dense-progression}]
Consider the shift operator $f:\Z \rightarrow \Z$, defined by $f(b)=b+h$,
and its iterations.
For given $b \in \Z$,
let $r(b)$ denote the least non-negative integer $r$ with 
\[
\{b, f(b), \ldots ,f^{r-1}(b)\} \subset  B, \text{ but }f^r(b)=b+rh\not\in B.\]
The  assumption $|\{b \in B:  b+h \in B\}| > (1-\alpha)|B|$
implies that for each fixed non-negative integer $r$, 
there are less than $\alpha |B|$ elements $b \in B$ with
this given value $r=r(b)$. Hence, for $k \in \N$
the number of elements $b\in B$
with $r(b) \leq k$ is less than 
$k \alpha |B| $.
If $k \alpha |B| \le |B|$, then
there exists a $b \in B$ with 
$r(b)\geq k+1$. For $k=\lfloor \frac{1}{\alpha}\rfloor$ this gives
$r(b) \geq \lfloor \frac{1}{\alpha}\rfloor +1$. 

Therefore $B$ contains
the arithmetic progression $\{b, b+h ,\ldots, b+(r-1)h\}$
of length $r\geq \lfloor \frac{1}{\alpha}\rfloor +1$. 
\end{proof}
\begin{proof}[Proof of Lemma \ref{growth}]
Let
\[
  c = \frac{k}{k-1}
\]
and
\[
  H_i=a_0+\{0,a_1\} + \cdots + \{0,a_i\}.
\]
Let us first assume that $a_0 \neq 0$.
Suppose that $|H|< 2c^{d-1}$.
Then there is some $i \in \{1, \ldots , d-1\}$ such that
\[
  \frac{|H_{i+1}|}{|H_i|}<c.
\]
Noting that
\[
  |H_{i+1}|=|H_i+\{0,a_{i+1}\}|=2|H_i|-|(H_i+a_{i+1}) \cap H_i| < c |H_i|,
\]
for this $i$ we find
\[
  |(H_i+a_{i+1})\cap H_i|> (2-c)|H_i|.
\]
Then by Lemma \ref{dense-progression} and our definition of $c$ the set 
$H_i$ contains an arithmetic progression of length 
\[
  \lfloor \frac{1}{c-1}\rfloor +1 = k
\]
which is a contradiction, as $S$ does not contain a progression of length $k$.\\
Finally, if $a_0=0$, we argue as above, first ignoring that the empty sum is
by definition not part of the Hilbert cube, but then deduct it by means of the
$-1$ expression in the statement of the result.

\end{proof}

As in our previous work
\cite{DietmannandElsholtz:2012}
 we make use of the following two results on squares:

\begin{lemma}[Theorem 9 of Gyarmati \cite{Gyarmati:2001}]
\label{kati}
Let $S_2$ denote the set of integer squares.
For sufficiently large $N$ the following holds true: 
If $C, D \subset \{1, \ldots, N\}$ such that $C+D\subset S_2$, then
\[
  \min\{|C|, |D|\} \leq 8 \log N.
\]
\end{lemma}
\begin{lemma} [Fermat, Euler
(see Volume II, page 440 of \cite{Dickson:1966})]{\label{Fermat}}
There are no four integer squares in arithmetic progression.
\end{lemma}
\begin{corollary}{\label{cor:fermatshift}}
Let $f(x)=ax^2+bx+c$ be a quadratic polynomial where $a,b,c \in \Z$
such that $a>0$, and let $S=\{f(x) : x \in \Z\}$.
Then the set $S$ does not contain
 four integer squares in arithmetic progression.
\end{corollary}

Let us remark that 
Setzer \cite{Setzer:1979} proved this corollary using elliptic curves.
Here we show that it is a simple consequence of Lemma
\ref{Fermat}, and therefore allows for an elementary proof.
\begin{proof}[Proof of Corollary \ref{cor:fermatshift}]
From $4af(x)+b^2-4ac=(2ax+b)^2$,
it follows that if the arithmetic four-progression
$P=\{n,n+m, n+2m, n+3m\}$ is contained in $S$, 
then the dilated and translated four-progression
$\{4an+b^2-4ac, 4an+b^2-4ac+4am,4an+b^2-4ac+8am,4an+b^2-4ac+12am\}$
 will be in the set of squares, contradicting Lemma \ref{Fermat}.
\end{proof}
\begin{proof}[Proof of Theorem \ref{thm:squares}]
Let
\[
  C=H(a_0; a_1, \ldots , a_{\lfloor d/2\rfloor})
\]
and
\[
  D= H(0; a_{\lceil (d+1)/2\rceil}, \ldots , a_d).
\]
By Lemma \ref{Fermat},
as the sets $C$ and $D$ are a subset, or a shifted subset,
of the set of squares, $C$ and $D$ do not contain an arithmetic
 progression of length $k=4$. 
Thus, by Lemma \ref{growth} and Lemma \ref{kati}, applied with $k=4$ to both
$C$ and $D$ individually, we obtain
\[c^{\lfloor d/2\rfloor }-1\leq \min( |C|, |D|)\leq 8 \log N,\]
where $c=4/3$. Hence
\[ d \leq \frac{2}{\log c} \log \log N+O(1)
\leq  6.96\log \log N,\]
for sufficiently large $N$.
\end{proof}
\begin{proof}[Proof of Theorem \ref{thm:kprogression}]
The proof is as the one above, except that
we replace Lemma \ref{kati} by the following 
lemma, due to
 Croot, Ruzsa and Schoen (\cite{CrootandRuzsaandSchoen:2007}, Corollary 1):
\begin{lemma}{\label{lem:croot}}
If $C,D\subset[1,N]$ and $C+D$ does not contain an arithmetic progression
of length $k$, then  
\[
  \min( |C|, |D|)\leq \sqrt{6}N^{1-\frac{1}{k-1}}.
\]
\end{lemma}
It then follows that
\[ d \leq \frac{2}{\log c} \frac{k-2}{k-1}\log N+\frac{2}{\log c}+1,\]
for sufficiently large $N$.
With $c=\frac{k}{k-1}$ the theorem follows.
\end{proof}
\begin{remark}
There are a number of results in the same spirit as Lemma \ref{lem:croot}.
Results of Green \cite{Green:2002}, Croot, \L aba and Sisask 
\cite{CrootandLabaandSisask:2013}
and Henriot \cite{Henriot:2013} lead to slightly stronger estimates, when $k$ is large.
However, it appears, in our application to
Theorem \ref{thm:kprogression} this would not even improve the ``2'' in the
the factor $2+o_k(1)$ so that we did not pursue this path.
\end{remark}

\begin{proof}[Proof of Corollary \ref{cor:quadraticpoly}]
From $4af(x)=(2ax+b)^2-b^2+4ac$,
it follows that if
\[
  H(a_0; a_1, \ldots, a_d) \subset S \cap [1,N],
\]
then
\begin{align*}
  4a H(a_0; a_1, \ldots, a_d) + b^2-4ac & \subset S_2 \cap
  [4a+b^2-4ac, 4aN+b^2-4ac] \\
  & \subset S_2 \cap [1, 4aN+b^2-4ac].
\end{align*}
Moreover,
\[
  4a H(a_0; a_1, \ldots, a_d) + b^2-4ac 
  = H(4a a_0+b^2-4ac; 4a a_1, \ldots, 4a a_d).
\]
The Corollary now follows immediately from the proof of Theorem 
\ref{thm:squares}, and observing that 
\[d < 6.96\log \log (4aN+b^2-4ac) + O(1) \leq 7 \log \log N\]
for sufficiently large $N$.
\end{proof}

\section{Hilbert cubes in powerful numbers, 
Proof of Theorem \ref{thm:Hilbert-powerful}
and Corollary  \ref{cor:Hilbert-pure}}

We first give an upper bound
on the maximal length of an arithmetic progression in the 
set $W$ of powerful numbers.
Let 
\[b,b+d, b+2d, \ldots, b+(k-1)d \in W\cap[1,N].\]

Let $p$ be a prime with $(d,p)=1$. Observe that
$k \geq 2p$ would imply that the progression contains at least two elements
which are divisible by $p$, and at least one of these is not divisible 
by $p^2$.
Hence, if $k \geq 2p$, then $p$ divides $d$.
From this it follows that for a  progression of length $k$
the difference $d$ is divisible by all primes $p \leq k/2$.
As $b+(k-1)d\leq N$ it follows that 
\[\prod_{p \leq k/2} p \leq d \leq \frac{N}{k-1}.\]
Taking logarithms it follows by the prime number theorem (with logarithmic
weight)
that
\[ \sum_{p \leq k/2} \log p =\frac{k}{2}+o(k) \leq \log d \leq \log N.\]
Hence $k \leq (2+o(1))\log N$.

Then making use of Theorem \ref{thm:kprogression}
with $k=(2+o(1))\log N$ proves that
\[
  d \le (4+o(1)) (\log N)^2
\]
for sufficiently large $N$, which confirms Theorem
\ref{thm:Hilbert-powerful}. Since $V \subset W$, Corollary
\ref{cor:Hilbert-pure} then follows immediately.

\section{Proof of Theorem \ref{brunnen}}

\subsection{Preparations}

\begin{lemma}[Erd\H{o}s and Shapiro, p.~862 of \cite{ES}]
\label{pacific}
Let $p$ be a prime and $\chi$ be a multiplicative character modulo $p$
that is not the main character. Moreover, let
$\overline{A}$ and $\overline{B}$ be subsets of $\Z/p\Z$. Then
\[
  \left| \sum_{a \in \overline{A}} \sum_{b \in \overline{B}}
  \chi(a+b) \right| \le (p |\overline{A}| |\overline{B}|)^{1/2}.
\]
\end{lemma}

\begin{lemma}[Gallagher's larger sieve, 
\cite{Gallagher:1971}]{\label{Gallagher}}
\label{gallagher}
Let ${\mathcal{S}}$ denote a set of primes 
and ${\mathcal{A}}\subset[1,N]$ such that for
$p \in {\mathcal{S}}$ the set ${\mathcal{A}}$ lies
modulo $p$ in at most $\nu(p)$
residue classes. Then the following bound holds,
provided the denominator is positive:
\[ |{\mathcal A}| \leq \dfrac{ - \log N + \sum_{p \in {\mathcal{S}}} \log p}{- \log N
+ \sum_{p \in {\mathcal{S}}} \dfrac{ \log p}{\nu(p)} }.
\]
\end{lemma}

\begin{lemma}[Quantitative version of Linnik's theorem, see Corollary 18.8 
in \cite{IK}]{\label{Linnik}}
There exists a positive constant $L$ such that
if $N\geq Q^L$, then
\[
  \sum_{\substack{p\leq N\\ p \equiv 1 \bmod Q}} \log p \gg
  \frac{N}{\varphi(Q)\sqrt{Q}},
\]
where the implied constant is absolute.
\end{lemma}

\subsection{A cube lemma for homogeneous arithmetic progressions}
In this subsection we briefly recall some definitions from \cite{DietmannandElsholtz:2012}
and then collect some results analogous to Theorem
\ref{thm:kprogression},
but more suitable for a setting where only information about homogeneous
arithmetic progressions is available, thus confining the applications to
subset sums rather than to general Hilbert cubes. Let us first remark
that by a \textit{homogeneous arithmetic progression of length $v$}
we mean a set
$\{d, 2d, \ldots, vd\}$ for suitable $d, v \in \N$.
In the following, let $S$ be any set of integers; for our intended
application, $S$ will later be specialised to the set $V$ of pure
powers.
Moreover, let $A=\{a_1, \ldots, a_d\}$ be a set of distinct positive
integers. We define $f(N)$ to be the minimum upper bound
such that whenever $B_1, \ldots, B_5 \subset \N$ such that
\begin{equation}
\label{november}
  B_i+B_j \subset S \cap [1, N]
\end{equation}
for all distinct $i, j \in \{1, \ldots, 5\}$, then
\[
  \min_{1 \le i \le 5} |B_i| \le f(N).
\]
Further, let $r_{A, h}(n)$ be the number of ways of writing $n$ as a sum
of exactly $h$
elements in $A$, where the order of the summands is irrelevant, and let
\[
  g(h, N)=\max_{n \le N} r_{A, h}(n).
\]
It turns out that $g(h, N)$ can be uniformly bounded in terms of the length of 
the longest
homogeneous arithmetic progression in $S$; note that this bound 
is independent of $A$.
\begin{lemma}
\label{progressionfreelemma}
Let $A=\{a_1, \ldots , a_d\} \subset [1,N]$ be a set of distinct integers,
and let 
$H(0;a_1, \ldots, a_d) \subset S \cap [1,N]$,
where $S$ is a set of positive integers
without a homogeneous arithmetic progression of length $v\geq 3$.
Then
\[
  g(h,N) \le h! (v-1)^{h-1}.
\]
\end{lemma}
The proof of this lemma is exactly the same as for Lemma 5 in
\cite{DietmannandElsholtz:2012}, taking $a_0=0$ throughout.
Note that with $a_0=0$ all arithmetic progressions
occurring in that proof turn out to be homogeneous,
and note that we need to replace $v-2$ in Lemma 5 in
\cite{DietmannandElsholtz:2012} by $v-1$ here, as with $a_0=0$ the
progressions of the form $0,d,2d,\ldots,(v-1)d$ of length $v$ occurring
there would reduce to a homogeneous progression $d,2d,\ldots,(v-1)d$ of
length $v-1$ here.\\
Our second main tool is the following result.
\begin{lemma}
\label{vancouver}
Let $A=\{a_1, \ldots , a_d\}\subset [1,N]$ be a set of distinct integers,
and suppose that $H(0; a_1, \ldots, a_d)
\subset S \cap [1,N]$. If $d \ge 5h+4$ for any $h \in \N$, then
\[
  d \le 5(h! f(N) g(h, N))^{1/h} + 5h + 4.
\]
\end{lemma}
\begin{proof}
This is basically the special case $m=1$ and $t=5$ of Theorem 4
in \cite{DietmannandElsholtz:2012}, but since our function
$f(N)$ is defined slightly differently from $f(5, N)$ there, let
us briefly outline the argument. Writing
$h \hat{\ } S$ for the $h$-fold
disjoint sumset of a set $S$, and letting $B_i= h \hat{\ } A_i$,
where $A = A_1 \cup \ldots \cup A_5 \cup R$ for mutually disjoint sets
$A_1, \ldots, A_5, R$ with $|A_1|=|A_2|=|A_3|=|A_4|=|A_5|
=\lfloor \frac{d}{5}\rfloor$
and $|R| \le 4$, we
first note that as in step 3 in the proof of Theorem 4 in
\cite{DietmannandElsholtz:2012}, we have
\begin{equation}
\label{ineq1}
  |B_i| \ge \frac{(|A_i|-h)^h}{h! g(h, N)}.
\end{equation}
Now by assumption $H(0; a_1, \ldots, a_d) \subset S \cap [1, N]$, whence
\eqref{november} is satisfied, so at least one $B_i$ satisfies the bound
\begin{equation}
\label{ineq2}
  |B_i| \le f(N).
\end{equation}
Rearranging \eqref{ineq1} and \eqref{ineq2} and noting that
$|A|=d$ then yields the result.
\end{proof}

\subsection{Sumsets in pure powers}

The following is close in spirit to Theorem 2 of 
Gyarmati, S\'{a}rk\"{o}zy and Stewart
\cite{GSS:2003}.
They
proved an upper bound on $|A|$, assuming that $A+'A\subset V$
where $A+'A = \{a + a': a, a' \in A, a \ne a'\}$.
Here we need a lemma for sums of distinct sets.
The following result about sums of five sets suffices for our purposes.
\begin{proposition}
\label{prop:tripleset}
There exists a positive constant $M$ with the following property:
If $A_1, \ldots, A_5 \subset [1,N]$ and $V$ denotes the set of pure powers,
and if
\[
  A_i + A_j \subset V
\]
for all distinct $i, j \in \{1, \ldots, 5\}$, then 
\[
   \min_{1 \le i \le 5} |A_i| \ll (\log N)^M.
\]
\end{proposition}
For ease of exposition,
we do not attempt to prove the best possible exponent $M$, since it is
irrelevant for our application of Proposition \ref{prop:tripleset},
but following  
the more detailed argument in the proof of Theorem 2
in \cite{GSS:2003}, one could for example obtain the explicit
value $M=48$.

\begin{proof}
We follow the strategy of the proof of Theorem 2 in \cite{GSS:2003}
but make some necessary adjustments.
Let
\[
  A_i=\{a_{i1} , \ldots , a_{it_{i}} \} \quad (1 \le i \le 5),
\]
and let us write each $a_{ij}$ in the unique form
\begin{equation}
\label{mini}
  a_{ij}=2^{u_{ij}}(4s_{ij}+e_{ij}),
\end{equation}
where the $u_{ij},s_{ij}$ are non-negative integers and
$e_{ij}\in \{-1,1\}$. Note that all
\[
  u_{ij} \le \left \lfloor \frac{\log N}{\log 2} \right \rfloor.
\] 
For $i \in \{1, \ldots, 5\}$, $u\in \{0, 1, \ldots , \lfloor \log N/\log 2 \rfloor\}$
and $e \in \{-1,1\}$ we define
\[
  A_{i,u,e}=\{a_{ij} \in A_i: u_{ij}=u, e_{ij}=e\}.
\]
With this definition we have
\[
  A_i=\bigcup_{u \le \lfloor \log N/\log 2 \rfloor} \quad 
  \bigcup_{e \in \{-1,1\}} A_{i,u,e}.
\]
For a given $i$, choosing the parameters $(u_i,e_i)$ such that
$|A_{i,u_i,e_i}|$ have maximal cardinalities, it follows that 
\begin{align*}
  |A_i| \ll |A_{i, u_i, e_i}| \log N \quad (1 \le i \le 5). 
\end{align*}
If at least three of the five $u_i$ are zero, 
then we can without loss of generality
assume that $u_1=u_2=0$ and $e_1=e_2$. On the other hand,
 if at most two of the five
$u_i$ are zero, then we can without loss of generality assume that
$u_1>0$, $u_2>0$ and $e_1=e_2$.
In either case, for $a_1 \in A_1$ and $a_2 \in A_2$ we obtain
\[
  a_1+a_2=2^{u_1}(4s_1+e_1)+2^{u_2}(4s_2+e_2)
  =2^g(2z+1)=y^q,
\]
where
\[
  g=g(u_1,u_2)=\begin{cases}
    u_1+1 & \text{if } u_1=u_2\\
    \min\{u_1, u_2\} & \text{if } u_1 \ne u_2
  \end{cases}
\]
is positive, $s_1$, $s_2$, $z$ and $y$ are suitable
integers, and $q$ is a prime dividing $g(u_1,u_2)$. Let
\[
  Q(u_1,u_2)=\prod_{q|g(u_1,u_2)} q,
\]
and observe that 
\begin{equation}{\label{q-upperbound}}
Q(u_1,u_2)\ll \log N.
\end{equation}
For a prime $p\equiv 1 \bmod Q(u_1,u_2)$, we write
$\overline{A}_p$ and $\overline{B}_p$ for the reductions of $A_{1,u_1,e_1}$
and $A_{2, u_2, e_2}$, respectively, modulo $p$.
The following result will be crucial
for the application of Gallagher's larger sieve.
\begin{lemma}
\label{sunshine}
Let $\varepsilon>0$. Further,
for fixed $(u_1,u_2)$, let $p$ be a prime with
$p \equiv 1 \bmod Q(u_1,u_2)$,
and let $\overline{A}$ and $\overline{B}\subset \Z/p\Z$ have the
following property:
If $a\in \overline{A}$ and $b\in \overline{B}$, then 
there is a prime $q$ dividing $g(u_1,u_2)$ and an element $y \in \Z/p\Z$,
such that $a+b\equiv y^q  \bmod p$. Then
\[
  |\overline{A}| |\overline{B}|
  \ll_{\varepsilon} p^{1+\varepsilon}.
\]
\end{lemma}
We postpone the proof of this lemma in order not to interrupt the
flow of our main argument.
By Lemma \ref{Linnik}
it follows that for $N_1\geq Q(u_1,u_2)^{L}$ we have
\begin{align}
\label{lower}
  \sum_{\substack{p \le N_1\\p \equiv 1 \bmod {Q(u_1,u_2)}}}
  \frac{\log p}{p^{1/2+\varepsilon}}  
  & \gg
  N_1^{-1/2-\varepsilon}
  \sum_{\substack{p \le N_1\\p \equiv 1 \bmod {Q(u_1,u_2)}}} \log p
  \nonumber \\
  & \gg
  \frac{N_1^{1/2-\varepsilon}}{\varphi(Q(u_1,u_2))\sqrt{Q(u_1,u_2)}}.
\end{align}
Moreover, for 
$N_1\geq Q(u_1,u_2)^{L}$ the Brun-Titchmarsh Theorem gives
\begin{align}
\label{bt}
  \sum_{\substack{p \le N_1\\p \equiv 1 \bmod {Q(u_1,u_2)}}}
  \log p & \le \log N_1
  \sum_{\substack{p \le N_1\\p \equiv 1 \bmod {Q(u_1,u_2)}}} 1
  \nonumber\\
  & \ll \log N_1 \frac{N_1}{\varphi(Q(u_1, u_2)) (\log N_1-\log Q)}
  \ll \frac{N_1}{\varphi(Q(u_1, u_2))}.
\end{align}
For $p \equiv 1 \bmod {Q(u_1,u_2)}$ write $\nu_A(p)$ for the
cardinality of $\overline{A}_p$, and $\nu_B(p)$ for the
cardinality of $\overline{B}_p$. Then
by Lemma \ref{sunshine} and \eqref{lower} we get
\[
  \sum_{\substack{p \le N_1\\p \equiv 1 \bmod {Q(u_1,u_2)}}}
  \frac{\log p}{\min\{\nu_A(p), \nu_B(p)\}}  
  \gg \frac{N_1^{1/2-\varepsilon}}{\varphi(Q(u_1,u_2))\sqrt{Q(u_1,u_2)}}.
\]
Without loss of generality we may assume that
\begin{equation}
\label{sonne}
  \sum_{\substack{p \le N_1\\p \equiv 1 \bmod {Q(u_1,u_2)}}}
  \frac{\log p}{\nu_A(p)}  
  \gg \frac{N_1^{1/2-\varepsilon}}{\varphi(Q(u_1,u_2))\sqrt{Q(u_1,u_2)}}.
\end{equation}
Applying Lemma \ref{gallagher}, the bounds
\eqref{q-upperbound},  \eqref{bt} 
and \eqref{sonne}, and choosing $N_1=(\log N)^{L+1}$
we obtain
\begin{align*}
  |A_{1,u_1,e_1}| & \leq \frac{-\log N +\sum_{
   p \le N_1, p \equiv 1 \bmod {Q(u_1,u_2)}} \log p}{- \log N+ 
 \sum_{p \le N_1, p \equiv 1 \bmod {Q(u_1,u_2)}} \frac{\log p}{\nu_A(p)}} \\
 & \ll \frac{-\log N + \frac{N_1}{\varphi(Q(u_1,u_2))}}{-
   \log N +\frac{N_1^{1/2-\varepsilon}}{\varphi(Q(u_1,u_2))\sqrt{Q(u_1,u_2)}}} \\
 & \ll N_1^{1/2+\varepsilon} \sqrt{Q(u_1,u_2)}.
\end{align*}
Since $Q(u_1, u_2) \ll \log N$, and $N_1=(\log N)^{L+1}$ it follows that
the hypothesis  $N_1 \geq Q(u_1,u_2)^L$ above is satisfied,
and we obtain
\[
  |A_{1,u_1,e_1}| \ll (\log N)^{L/2+2}.
\]
Hence
\[
   |A_1| \ll (\log N)^{L/2+3}.
\]
It remains to prove Lemma \ref{sunshine}. We closely follow the proof
of Lemma 2 in \cite{GSS:2003} and will therefore be rather brief.
It is convenient to write
\[
  Q=\{q_1, \ldots, q_K\}
\]
for the set of all primes dividing $g(u_1,u_2)$. Now let
\[
  F(n) = \begin{cases}
  1 & \mbox{if $x^q \equiv n \bmod p$ is solvable for some $q \in Q$}\\
  0 & \mbox{otherwise}.
  \end{cases}
\]
Moreover, let $F^*(n)$ be as introduced in formula (2.9) in
\cite{GSS:2003}, i.e.
\[
  F^*(n) = \sum_{\ell=1}^K (-1)^{\ell+1}
  \sum_{\substack{q_1 < \cdots < q_\ell \\ q_i \in Q}}
  \frac{1}{q_1 \cdots q_l}
  \sum_{j_1=0}^{q_1-1} \cdots \sum_{j_\ell=0}^{q_\ell-1}
  \chi_{q_1}^{j_1} \cdots \chi_{q_\ell}^{j_\ell}(n),
\]
where the $\chi$ are suitable multiplicative characters modulo $p$;
note that we make use of our assumption $p \equiv 1 \mod {Q(u_1, u_2)}$
here.
Observe that $F(n)=F^*(n)$, if $(n, p)=1$. On the other hand,  $F(n)=1$ but
$F^*(n)=0$, when $n$ is divisible by $p$. Consequently,
\begin{align*}
  \sum_{a \in \overline{A}} \sum_{b \in \overline{B}} F(a+b)
  & \le \sum_{a \in \overline{A}} \sum_{b \in \overline{B}} F^*(a+b) +
  \min\{|\overline{A}|, |\overline{B}|\} \\
  & \le \sum_{a \in \overline{A}} \sum_{b \in \overline{B}} F^*(a+b) 
  + (|\overline{A}| |\overline{B}|)^{1/2}.
\end{align*}
As shown on page 6 of \cite{GSS:2003},
$\chi_{q_1}^{j_1} \cdots \chi_{q_\ell}^{j_\ell}$ is the main
character if and only if $j_1 = \ldots = j_\ell = 0$.
Thus
\[
  \sum_{a \in \overline{A}} \sum_{b \in \overline{B}} F^*(a+b)
   = {\sum_1 \nolimits + \sum_2 \nolimits},
\]
where
\begin{align*}
  {\sum_1 \nolimits} & = \sum_{\ell=1}^K (-1)^{\ell+1}
  \sum_{\substack{q_1 < \ldots < q_\ell \\ q_i \in Q}}
  \frac{1}{q_1 \cdots q_\ell}
  \sum_{\substack{a \in \overline{A}, b \in \overline{B} \\ p \nmid a+b}} 1\\
  & = \left(1-\prod_{q \in Q} (1-\frac{1}{q}) \right)
  \sum_{\substack{a \in \overline{A}, b \in \overline{B} \\ p \nmid a+b}}
  1 \\
  & \le
  \left(1-\prod_{q \in Q} (1-\frac{1}{q}) \right) |\overline{A}| |\overline{B}| \\
\end{align*}
and
\[
  {\sum_2 \nolimits} = \sum_{\ell=1}^K (-1)^{\ell+1}
  \sum_{\substack{q_1 < \ldots < q_\ell \\ q_i \in Q}}
  \frac{1}{q_1 \cdots q_l}
  \sum_{\substack{0 \le j_1 \le q_1, \ldots, 0 \le j_\ell \le q_\ell \\
  (j_1, \ldots, j_\ell) \ne (0, \ldots, 0)}} \;
  \sum_{a \in \overline{A}} \sum_{b \in \overline{B}}
  \chi_{q_1}^{j_1} \cdots \chi_{q_\ell}^{j_\ell}(a+b).
\]
As remarked above, here the characters
$\chi_{q_1}^{j_1} \cdots \chi_{q_\ell}^{j_\ell}$ are all different from
the main character, whence Lemma \ref{pacific} gives
\[
  \left| \sum_{a \in \overline{A}} \sum_{b \in \overline{B}}
  \chi_{q_1}^{j_1} \cdots \chi_{q_\ell}^{j_\ell}(a+b) \right|
  \le (p |\overline{A}| |\overline{B}|)^{1/2}.
\]
Hence
\[
  |{\sum_2 \nolimits}| \le (p |\overline{A}| |\overline{B}|)^{1/2}
  \sum_{\ell=1}^K \sum_{\substack{q_1 < \cdots < q_\ell \\ q_i \in Q}} 1
  \le (p |\overline{A}| |\overline{B}|)^{1/2} 2^K.
\]
Note that $K=\omega(g(u_1,u_2))$,
where $\omega(t)$ denotes the number of distinct prime factors of $t$.
Summarising the above, we obtain
\begin{align*}
  \sum_{a \in \overline{A}} \sum_{b \in \overline{B}} F(a+b) & \le
  \left( 1-\prod_{q | Q(u_1,u_2)} (1-\frac{1}{q}) \right)
  |\overline{A}| |\overline{B}| \\
  & + 2^{\omega(g(u_1,u_2))} (p |\overline{A}| |\overline{B}|)^{1/2} +
  (|A| |B|)^{1/2}.
\end{align*}
On the other hand,
\[
  \sum_{a \in \overline{A}} \sum_{b \in \overline{B}} F(a+b)
  = |\overline{A}| |\overline{B}|
\]
by our assumption that for all $a \in \overline{A}$ and $b \in \overline{B}$
there exists a prime $q \in Q$ such that $a+b \equiv x^q \bmod p$
for some integer $x$.
Consequently,
\[
  \prod_{q | Q(u_1,u_2)} \left(1-\frac{1}{q} \right)
  (|\overline{A}| |\overline{B}|)^{1/2}
  \le p^{1/2} 2^{\omega(g(u_1,u_2))} + 1
\]
holds. Let us recall that by the definition of $Q(u_1,u_2)$, we have
$\omega(g(u_1,u_2))=\omega(Q(u_1,u_2))$, and in
general $\omega(n)\leq 2 \frac{\log n}{\log \log n}$.
Moreover, $p>Q(u_1,u_2)$, as $p\equiv 1 \bmod Q(u_1,u_2)$.
The result thus follows on noting that
\[
  2^{\omega(g(u_1,u_2))} = 2^{\omega(Q(u_1, u_2))} \leq
  2^{2\log Q(u_1, u_2)/\log \log Q(u_1,u_2)} \leq 
 2^{2\log p/\log \log p} \ll_{\varepsilon} p^{\varepsilon} 
\]
and
\[
  \prod_{q | Q(u_1,u_2)} \left(1-\frac{1}{q} \right)^{-1} \ll
  \exp(\sum_{q \le Q(u_1,u_2)} \frac{1}{q}) \ll \log Q(u_1,u_2) \le \log p
\]
by Merten's Theorem.
\end{proof}
\subsection{Subsetsums in pure powers}
Recall that $V$ denotes the set of pure powers.
For $N \in \N$ write $\ell(N)$ for the maximum
positive integer $\ell$ such that there exists $x \in V$ such that
all elements of the arithmetic progression $x, 2x, \ldots, \ell x$ are
elements of $V \cap [1,N]$.
\begin{lemma}
\label{lem:progressionlength}
For sufficiently large $N$, we have
\[
  \frac{\log \log N}{\log \log \log N} \ll \ell(N)
  \ll \frac{\log \log N}{\log \log \log N}.
\]
\end{lemma}
\begin{proof}
For the lower bound, see the proof of Theorem 3 in \cite{GSS:2003}.
The proof of the upper bound starts with the observation that $x$ cannot
be odd, since otherwise $2x \equiv 2 \bmod 4$ would contradict $2x \in
V$. Therefore it follows that
 $2^\alpha || x$, where $\alpha \ge 1$. Then also
$2^\alpha || nx$ for all odd $n$ where $n \le \ell$. In particular,
from $nx \in V$ we conclude that $nx$ must be a $q$-th power
of an integer for some prime $q$ dividing $\alpha$. Now $2^\alpha | x$
where $x \le N$, so $\alpha \ll \log N$, whence there are only
$$\omega(\alpha)\ll \frac{\log \alpha}{\log \log \alpha}=
O\left(\frac{\log \log N}{\log \log \log N}\right)$$ possibilities for $q$.
An elementary but useful
 observation is: if $n_1, n_2 \in \N$ are both square-free
and $n_1x$ and $n_2x$ are both a $q$-th power for some prime $q \ge 2$,
then necessarily $n_1=n_2$: Write $n_1x=a^q$ and $n_2x=b^q$ for suitable
$a, b \in \N$. Then
\[
  \frac{n_1}{n_2} = \left(\frac{a}{b}\right)^q.
\]
Since $\frac{n_1}{n_2}$ is square-free, $\sqrt[q]{\frac{n_1}{n_2}}$ is
irrational unless $n_1=n_2$.
Now write
\[
  \mathcal{N}=\{n \in \N: n \le \ell, n \equiv 1 \bmod 2 \mbox{ and $n$
  is square-free}\}.
\]
As shown in the observation above, if $n_1, n_2\in \mathcal{N}$ 
are distinct elements, 
then $n_1x$ and $n_2x$ are a $q_1$-th power and a
$q_2$-th power of a suitable integer, respectively, where $q_1$ and $q_2$
are distinct primes. On the one hand,
it is well known that $|\mathcal{N}| \gg \ell$. On the other hand, we have seen
that there are only $O\left(\frac{\log \log N}{\log \log \log N}\right)$
possibilities for those prime exponents $q_1, q_2$. Therefore,
$\ell = O\left(\frac{\log \log N}{\log \log \log N}\right)$.
\end{proof}

\begin{proof}[Proof of Theorem \ref{brunnen}]
By Lemma \ref{progressionfreelemma} and Lemma \ref{lem:progressionlength}
we have
\[
  g(h, N) \le h! \left( c_1 \frac{\log \log N}{\log \log \log N}
  \right)^{h-1}
\]
for a suitable absolute constant $c_1$.
Proposition \ref{prop:tripleset} then reads
\[
  f(N) \le c_2 (\log N)^M
\]
for suitable absolute constants $c_2$ and $M$. Now Lemma
\ref{vancouver} gives
\[
  d \le 5 \left(h! c_2 (\log N)^M h!
  \left( c_1 \frac{\log \log N}{\log \log \log N} \right)^{h-1}
  \right)^{1/h} + 5h + 4.
\]
Choosing $h=\log \log N$ and using the trivial bound $h! \le h^h$
yields the claimed result
$d=O\left(\frac{(\log \log N)^3}{\log \log \log N}\right)$.

When allowing repeated values of the $a_i$, it is clear that the 
maximal multiplicity is 
$m=O(\frac{\log \log N}{\log \log \log N})$, as the set of 
subset sums of $m$ times the same element $a_i$ generates 
a homogeneous arithmetic progression of length 
$m\leq \ell(N) = O(\frac{\log \log N}{\log \log \log N})$.

\end{proof}

\section{Proofs of Theorems \ref{infinitecubeV} and
\ref{infinitecubeW}}

Let us first prove Theorem \ref{infinitecubeV}.
We can adapt the proof of Theorem 4 of \cite{GSS:2003}.
It is enough to show that there are no five infinite sets
$A_1, \ldots, A_5 \subset \mathbb{Z}$ such that
$A_1 + \cdots + A_5 \subset V$.
To this end we write the elements of each set $A_i$ again in the
form \eqref{mini}. By relabeling $A_1, \ldots, A_5$ if necessary,
as in the proof of Proposition \ref{prop:tripleset} we can without
loss of generality assume that $e_1=e_2$ and one of the following is
true:\\
\emph{Case 1:} Both sequences $u_{1j}$ and $u_{2j}$ are bounded.\\
\emph{Case 2:} Both sequences $u_{1j}$ and $u_{2j}$ are unbounded.\\
In each case we can finish the argument as in Case 1 or Case 2,
respectively, in the proof of Theorem 4 of \cite{GSS:2003}.\\
Next, let us establish Theorem \ref{infinitecubeW}.
Without loss of generality we may assume that all the $a_i$ are distinct.
For, otherwise, some element $a$ occurs infinitely often,
generating an infinite arithmetic progression, which contradicts
the observation that arithmetic progressions can only be short compared 
to the size of any finite interval $[1,N]$,
i.e. $k \leq (2+o(1))\log N$, made above in the proof of 
Theorem \ref{thm:Hilbert-powerful}.

Then first note that
Granville (\cite{Granville:1998}, Theorem 7) proved that the $abc$
conjecture implies that
\[
  \lim_{n \rightarrow \infty} (t_{n+2} - t_n) = \infty,
\]
where $t_n$ denotes the $n$-th powerful number.
The existence of an infinite Hilbert cube in $W$ would imply
that the following values are all powerful, for $i=3,4, \ldots $:
\[ a_0+a_i< a_0 +a_1 +a_i < a_0+a_1+a_2+a_i,\]
 but their difference is bounded, a contradiction.
Finally, observe that
$P(x)=x(x+a_1)(x+a_1+a_2)$ has three distinct simple zeros. 
For every $x\in H(a_0;a_3,a_4, \ldots )$
all elements $x, x+a_1$ and $x+a_1+a_2$ are powerful,
so that the product is
powerful, too. But, assuming Conjecture \ref{red2},
there are only finitely many possible values for $x$.


\begin{thebibliography}{99}

\bibitem{AlonandAngelandBenjaminiandLubetzky:2012}
N.~Alon, O.~Angel, I.~Benjamini, E.~Lubetzky,
Sums and products along sparse graphs.
Israel Journal of Mathematics 188 (1) (2012), 353--384.

\bibitem{Bergelson:2010}
V.~Bergelson,
Ultrafilters, IP sets, Dynamics, and Combinatorial Number Theory. 
Contemporary Mathematics 530 (2010), 23--47. 

\bibitem{BrownandErdosandFreedman:1990}
T.C.~Brown, P.~Erd\H{o}s, A.R.~Freedman,
Quasi-progressions and descending waves. 
J. Comb. Theory, Ser. A 53 (1990), no.~1, 81--95.

\bibitem{CillerueloandGranville:2007}
J.~Cilleruelo, A.~Granville,
Lattice points on circles, squares in arithmetic progressions, 
and sumsets of squares.
In ``Additive Combinatorics", CRM Proceedings \& Lecture Notes,
vol. 43 (2007), 241--262.

\bibitem{ConlonFoxSudakov}
D.~Conlon, J.~Fox and B.~Sudakov, 
Short proofs of some extremal results.
Combinatorics, Probability and Computing
23 (2014), 8--28.

\bibitem{CrootandLabaandSisask:2013}
E.~Croot, I.~\L aba,  O.~Sisask, 
Arithmetic progressions in sumsets and $L^p$-almost-periodicity.
Combin. Probab. Comput. 22 (2013), no. 3, 351--365. 

\bibitem{CrootandRuzsaandSchoen:2007}
E.~Croot, I.~Ruzsa, T.~Schoen,
Arithmetic progressions in sparse sumsets.
Integers 7, No. 2, Paper A10 (2007).

\bibitem{Csikvari:2008}
P.~Csikv\'{a}ri,
Subset sums avoiding quadratic nonresidues.
Acta Arith. 135 (2008), no. 1, 91--98. 

\bibitem{DarmonandMerel:1997}
H.~Darmon, L.~Merel:
Winding quotients and some variants of Fermat's last theorem. 
J. Reine Angew. Math. 490 (1997), 81--100.

\bibitem{Dickson:1966}
L.E.~Dickson, History of the theory of numbers, 
Vol. II: Diophantine analysis; 
reprint: Chelsea Publishing Co., New York 1966.

\bibitem{DietmannandElsholtz:2012}
R.~Dietmann, C.~Elsholtz, 
Hilbert cubes in progression-free sets and in the set of squares.
Israel Journal of Mathematics, 
Israel J. Math. 192 (2012), 59--66.

\bibitem{DujellaandElsholtz:2013}
A.~Dujella, C.~Elsholtz, Sumsets being squares,
Acta Math. Hungar. 141 (2013), 353--357. 


\bibitem{Erdos:1981}
P.~Erd\H{o}s, On the combinatorial problems which I would most 
like to see solved. Combinatorica 1 (1981), no. 1, 25--42.

\bibitem{ErdosandRado:1960}
P.~Erd\H{o}s, R.~Rado,
Intersection theorems for systems of sets.
J. London Math. Soc. 35 (1960), 85--90. 

\bibitem{ErdosandSarkozy:1992}
P.~Erd\H{o}s, A.~S\'{a}rk\"{o}zy,
Arithmetic progressions in subset sums.
Discrete Math. 102 (1992), no. 3, 249--264. 

\bibitem{ES}
P.~Erd\H{o}s, H.N.~Shapiro,
On the least primitive root of a prime.
Pacific J. Math. 7 (1957), 861--865.

\bibitem{Erdos-Szemeredi:1983}
P.~Erd\H{o}s, E.~Szemer\'{e}di,
On sums and products of integers. 
Studies in pure mathematics, 213--218, Birkh\"{a}user, Basel, 1983. 

\bibitem{ErdosandTuran:1936}
P. Erd\H{o}s, P. Tur\'{a}n,
On some sequences of integers, J. London Math. Soc. 11 (1936), 261--264.

\bibitem{Gallagher:1971} 
P.X.~Gallagher, A larger sieve. Acta Arith. 18 (1971), 77--81.

\bibitem{Gowers:2011}
W.T.~Gowers, A new proof of Szemer\'{e}di's theorem.
Geom. Funct. Anal. 11 (2001), 465--588.

\bibitem{Granville:1998}
A.~Granville, ABC Allows Us to Count Squarefrees,
Internat. Math. Res. Notices 19 (1998), 991--1009.

\bibitem{Green:2002}
B.~Green, Arithmetic progressions in sumsets. Geom. Funct. Anal. 12 (2002), no. 3, 584--597.


\bibitem{GreenandTao:2008}
B.J.~Green and T.~Tao,
The primes contain arbitrarily long arithmetic progressions.
Annals of Mathematics 167 (2008), 481--547.

\bibitem{GundersonandRodlandSidorenko:1999}
D.S.~Gunderson, V.~R\"{o}dl and A.~Sidorenko, 
Extremal problems for sets forming Boolean algebras
and complete partite hypergraphs. J. Combin. Theory Ser. A 88 (1999), 342--367.

\bibitem{Gyarmati:2001}
K.~Gyarmati,  On a problem of Diophantus. Acta Arith.
97 (2001), 53--65.

\bibitem{GSS:2003}
K.~Gyarmati, A.~S\'{a}rk\"{o}zy, 
C.L. Stewart, On sums which are powers. Acta Math. Hungar.
99 (2003), 1--24.

\bibitem{GyoriHajduPinter:2009}
K.~Gy\"ory, L.~Hajdu and \'{A}.~Pint\'{e}r, 
Perfect powers from products of consecutive terms in arithmetic
progression.
Compositio Math. 145 (2009), 845--864.

\bibitem{Hegyvari:1996}
N.~Hegyv\'{a}ri, 
On representation problems in the additive number theory, 
Acta Math. Hungar. 72 (1996), 35--44.

\bibitem{Hegyvari:1999}
N.~Hegyv\'{a}ri, 
On the dimension of the Hilbert cubes. J. Number Theory 77 (1999), 326--330.

\bibitem{HegyvariandSarkozy:1999}
N.~Hegyv\'ari, A.~S\'ark\"ozy, On Hilbert cubes in certain sets. The
Ramanujan Journal 3 (1999), 303--314.

\bibitem{Henriot:2013} 
K.~Henriot, On arithmetic progressions in $A+B+C$,
Int. Math. Res. Notices (2014) 18: 5134--5164. 


\bibitem{Hilbert:1892}
D.~Hilbert, \"{U}ber die Irreduzibilit\"{a}t 
ganzer rationaler Funktionen mit ganzzahligen Koeffizienten.
J. Reine Angew. Math. 110 (1892), 104--129.

\bibitem{IK}
H.~Iwaniec, E.~Kowalski,
Analytic Number Theory, American Mathematical Society, Colloquium
Publications, 53. Providence (2004).

\bibitem{KhalfalahandSzemeredi:2006}
A.~Khalfalah,  E.~Szemer\'{e}di, 
On the number of monochromatic solutions of $x+y=z^2$.
Combin. Probab. Comput. 15 (2006), no. 1-2, 213--227.

\bibitem{Kostochka:1996}
A.V.~Kostochka,
An intersection theorem for systems of sets.
Random Structures Algorithms 9 (1996), no. 1-2, 213--221. 

\bibitem{LagariasandOdlyzkoandShearer:1983}
J.C.~Lagarias, A.M.~Odlyzko and J.B.~Shearer,
On the density of sequences of integers the sum of no two of which is a 
square. II. General sequences.
J. Combin. Theory Ser. A 34 (1983), no. 2, 123--139. 

\bibitem{Masser}
D.~W.~Masser, Open problems.
Proceedings of the Symposium on Analytic Number Theory,
London, Imperial College (1985).

\bibitem{NguyenandVu:2010}
H.H.~Nguyen,  V.H.~Vu, 
Squares in sumsets.  An irregular mind, 491--524, 
Bolyai Soc. Math. Stud., 21, J\'{a}nos Bolyai Math. Soc., Budapest, 2010.

\bibitem{RSS}
J.~Rivat, A.~S\'{a}rk\"{o}zy, C.L.~Stewart,
Congruence properties of the $\Omega$-function on sumsets.
Illinois J. Math. 43 (1999), no. 1, 1--18. 

\bibitem{Ruzsa:1991}
I.Z.~Ruzsa, Arithmetic progressions in sumsets.
Acta Arith. 60 (1991), no. 2, 191--202. 

\bibitem{Sarkozy:1994}
A.~S\'{a}rk\"{o}zy,
Finite addition theorems. II.
J. Number Theory 48 (1994), no. 2, 197--218. 

\bibitem{ST}
A.~Schinzel, R.~Tijdeman,
On the equation $y^m=P(x)$. Acta Arith. XXXI (1976), 199--204.

\bibitem{Schoen:2011}
T.~Schoen, Arithmetic progressions in sums of subsets of sparse sets.
 Acta Arith. 147 (2011), no. 3, 283--289.

\bibitem{Setzer:1979}
B.~Setzer,
Arithmetic progressions in the values of a quadratic polynomial.
Rocky Mountain J. Math. 9 (1979), no. 2, 361--363. 

\bibitem{Solymosi:2007}
J.~Solymosi, Elementary Additive Combinatorics. 
In ``Additive Combinatorics'', CRM Proceedings \& Lecture Notes,
vol.~43 (2007), 29--38.

\bibitem{Szemeredi:1969}
E.~Szemer\'{e}di,
On sets of integers containing no four elements in arithmetic progression.
Acta Math. Acad. Sci. Hung. 20 (1969), 89--104. 

\bibitem{SzemerediandVu:2006}
E.~Szemer\'{e}di, V.H.~Vu.
Finite and infinite arithmetic progressions in sumsets.
Ann. of Math. (2) 163 (2006), no. 1, 1--35. 

\bibitem{SzemerediandVu:2006JAMS}
E.~Szemer\'{e}di, V.H.~Vu,
Long arithmetic progressions in sumsets: thresholds and bounds.
J. Amer. Math. Soc. 19 (2006), no. 1, 119--169. 

\bibitem{Walsh:1999}
P.G.~Walsh, On a conjecture of Schinzel and Tijdeman. 
Number Theory in Progress
(proceedings of a conference in honour of the sixtieth birthday 
of A. Schinzel, Zakopane, Poland, 1997),
 p. 577-582, Walter de Gruyter, 1999.

\bibitem{Woods:2004}
A. R. Woods,
 Subset sum ``cubes'' and the complexity of primality testing,
Theoretical Computer Science
Volume 322 (1) (2004),  203--219.

\end{thebibliography}
\end{document}